\documentclass[11pt]{article}
\usepackage{mathrsfs}
\usepackage{amsmath, mathtools}
\usepackage{amssymb}
\usepackage{amsthm}
\usepackage{graphicx}
\usepackage{epic}
\usepackage{xcolor,cite}
\usepackage{cases}
\usepackage{pst-poly}
\usepackage{pst-plot}
\usepackage{CJK,makecell}
\usepackage{multirow}

\usepackage{verbatim}
\usepackage{subfig}
\usepackage{booktabs}
\definecolor{dkgreen}{rgb}{0,0.6,0}
\numberwithin{equation}{section}

\newcommand{\arxiv}[2]{\href{https://arxiv.org/abs/#1}{\texttt{arXiv:#1}} \texttt{[#2]}}

\usepackage{etex}

\usepackage[left,mathlines,displaymath]{lineno}

\usepackage{pgf}
\usepackage{listings}
\definecolor{mauve}{rgb}{0.58,0,0.82}
\definecolor{dkgreen}{rgb}{0,0.6,0}
\lstdefinestyle{pitonche} {
    language = Python,
    basicstyle = footnotesizettfamily,
    showspaces = false,
    showstringspaces = false,
    breakautoindent = true,
    flexiblecolumns = true,
    keepspaces = true,
    stepnumber = 1,
    xleftmargin = 0pt
}
\renewcommand{\paragraph}{\roman{paragraph}}

\usepackage{bm}
\usepackage[hidelinks]{hyperref}
\usepackage{tikz}
\usetikzlibrary{automata}
\usepackage{enumitem}
\usetikzlibrary{calc}
\usetikzlibrary{arrows,shapes,positioning}
\usetikzlibrary{decorations.markings}
\tikzstyle arrowstyle=[scale=1]
\tikzstyle directed=[postaction={decorate,decoration={markings, mark=at position .65 with {\arrow[arrowstyle]{stealth}}}}]
\tikzstyle reverse directed=[postaction={decorate,decoration={markings, mark=at position .65 with {\arrowreversed[arrowstyle]{stealth};}}}]

\newtheorem{claim}{Claim}[section]
\newtheorem{theorem}{Theorem}[section]
\newtheorem{corollary}[theorem]{Corollary}

\newtheorem{problem}[theorem]{Problem}
\newtheorem{lemma}[theorem]{Lemma}

\newtheorem{proposition}[theorem]{Proposition}

\renewcommand{\qedsymbol}{\scalebox{0.7}{$\blacksquare$}}

\usepackage{authblk}

\title{Maximum number of spanning trees in\\bipartite graphs with a given diameter%
\thanks{Email addresses: \texttt{xush0928@163.com} (S.~Xu), \texttt{ivan.damnjanovic@elfak.ni.ac.rs} (I.~Damnjanovi\'c), \texttt{kexxu1221@126.com} (K.~Xu).}}

\author[1,2]{Shaohan Xu}
\author[3,4]{Ivan Damnjanovi\'{c}\thanks{Corresponding author.}}
\author[1,2]{Kexiang Xu}

\affil[1]{School of Mathematics, Nanjing University of  Aeronautics and Astronautics,\linebreak Nanjing, Jiangsu, 210016, PR China}
\affil[2]{MIIT Key Laboratory of Mathematical Modelling and High Performance Computing of Air \linebreak Vehicles, Nanjing, Jiangsu, 210016, PR China}
\affil[3]{Faculty of Electronic Engineering, University of Ni\v{s}, Aleksandra Medvedeva 4, Ni\v{s}, 18104, Serbia}
\affil[4]{Faculty of Mathematics, Natural Sciences and Information Technologies, University of Primorska,\linebreak Glagolja\v{s}ka 8, Koper, 6000, Slovenia}

\date{}

\begin{document}

\maketitle

\begin{abstract}
The number of spanning trees is a classical graph invariant and an important measure of network reliability, as it counts the minimal connected spanning substructures that can maintain communication in a network. Let $\mathcal{B}(n,d)$ be the set of connected bipartite graphs of order $n$ and diameter $d$. Motivated by reliability design problems for bipartite network models with fixed order and diameter, this paper determines all graphs with the maximum number of spanning trees in $\mathcal{B}(n,d)$. The result gives an extremal characterization of bipartite network topologies with the largest number of connected spanning backbones under prescribed order and diameter constraints, and provides a structural reference for the design of reliable bipartite networks.
\end{abstract}


\medskip\noindent
\textbf{AMS Subj.\ Class.}: 05C05, 05C30, 05C35

\section{Introduction}
All graphs considered in this paper are finite, simple, and undirected. Let $G$ be a graph with vertex set $V(G)$ and edge set $E(G)$. For a vertex $v\in V(G)$, let $N_{G}(v)$ be the set of vertices adjacent to $v$ in $G$, where $d_{G}(v)=|N_{G}(v)|$ is the degree of $v$ in $G$. For any subset $S\subseteq V(G)$, $G[S]$ is the subgraph of $G$ induced by $S$. For any $V'\subseteq V(G)$ and $E'\subseteq E(G)$, let $G-V'$ (resp.\ $G-E'$) be the subgraph of $G$ obtained by deleting all vertices in $V'$ and all edges incident with them (resp.\ all edges in $E'$) . The {\it distance} between two vertices $u$ and $v$ in $G$ is denoted by $d_{G}(u,v)$. The {\it diameter} of $G$ is the maximum distance between any two vertices in $G$. Let $\mathcal{B}(n,d)$ be the set of connected bipartite graphs with diameter $d$ and  order~$n\ge d+1$. Let $\mathcal{T}(G)$ be the set of spanning trees of $G$, and let $\tau(G)=|\mathcal{T}(G)|$ be the number of spanning trees of $G$.

The enumeration of spanning trees in a graph is an important and popular problem in graph theory, which has connections with many other fields, including theoretical physics and computer science, among others. For example, the number of spanning trees $\tau(G)$ is a key parameter in Tutte polynomials \cite{T2,E12},  and it is also equal to the order of the critical group \cite{C5,D3} of the graph, also known as the sandpile group in statistical physics or the Picard group in algebraic geometry.  Moreover, the number of spanning trees has a close relationship with the partition function of the $q$-state Potts model in statistical mechanics \cite{W2, F1}. Several explicit formulas for counting spanning trees in some special classes of graphs have been established, such as graphs with symmetry \cite{Y1}, line graphs \cite{D1,G2}, Bruhat graphs \cite{R2}, Apollonian networks \cite{ZhangWuComellas14}, almost-complete multipartite graphs \cite{C3} and the $K_n$-complement of bipartite graphs \cite{G1}. Related electrical network transformations have also been used to derive resistance distances and spanning tree counting formulas in weighted multipartite graph models \cite{ChengGe26}. For other relevant results in the literature, the readers can refer to \cite{E1,K2,T1,Z1,Z2}.

The number of spanning trees also has a natural interpretation in network reliability. From the viewpoint of network theory, suppose that all nodes of a network $G$ are perfectly reliable and that all links operate independently with the same probability $p$. If $\zeta_i(G)$ denotes the number of connected spanning subgraphs of $G$ with $i$ edges, then the all-terminal reliability of $G$ is defined by
\[
R(G,p)=\sum_{i=n-1}^{m}\zeta_i(G)p^i(1-p)^{m-i},
\]
where $n=|V(G)|$ and $m=|E(G)|$. Since $\zeta_{n-1}(G)=\tau(G)$, the number of spanning trees is precisely the coefficient corresponding to connected spanning subgraphs with the minimum possible number of edges. When $p$ is small, the reliability polynomial is dominated by the term corresponding to spanning trees. Hence, the number of spanning trees is commonly used as a combinatorial indicator of network reliability, since a graph with more spanning trees has more minimal connected spanning structures available to maintain connectivity \cite{C2}. The spanning tree number reflects the connectivity robustness of a network. A larger value of this parameter indicates that the network has more alternative spanning backbones to maintain connectivity under edge failures.

From the perspective of extremal graph theory and network reliability, it is natural to determine the maximum or minimum number of spanning trees, together with the corresponding extremal graphs, among special classes of graphs. This problem has applications in experimental design \cite{C1} and network theory \cite{K1}. Several results on this extremal problem have been obtained. Grimmett \cite{G4}, Grone and Merris \cite{G5} and Das \cite{D2} established upper bounds on $\tau(G)$ for various graphs with given parameters. Bogdanowicz characterized the graphs that attain the minimum number of spanning trees in chordal graphs with a fixed vertex connectivity \cite{B7,B9} and in all connected graphs with $n$ vertices and $m$ edges \cite{B6,B8}. Ok and Thomassen \cite{O1} determined lower bounds on $\tau(G)$ for a $k$-edge-connected multigraph, together with the corresponding extremal multigraphs when $k$ is even. Pek\'arek, Sereni and Yilma \cite{P3} fully characterized all extremal multigraphs with the minimum number of spanning trees in connected $k$-regular multigraphs.

In addition, such extremal problems for the number of spanning trees have also been studied in bipartite graphs. Gong, Xu, Zou and Wang \cite{G8} characterized the graphs minimizing the number of spanning trees among bipartite graphs with a fixed cyclomatic number. All bipartite graphs with given vertex (resp.\ edge) connectivity and order maximizing the number of spanning trees were determined in \cite{X1}.
 Ehrenborg and van Willigenburg \cite{E1} introduced a class of bipartite graphs known as Ferrers graphs and provided an exact formula for their numbers of spanning trees. A central problem arising from their work, known as Ehrenborg's conjecture, asserted that for any connected bipartite graph $G$ with bipartition $V_1\cup V_2$, the inequality
\[
        \tau(G)\leq\frac{\prod_{v\in V(G)}d_{G}(v)}{|V_1||V_2|}
\]
holds. Ferrers graphs are known to attain this upper bound. Recently, Ho \cite{H1} completely resolved this conjecture, proving that the bound is valid and that equality occurs if and only if $G$ is a Ferrers graph.

Motivated by the above extremal problems for bipartite graphs and by the reliability interpretation of spanning trees, we study the spanning tree maximization problem in $\mathcal{B}(n,d)$. From the viewpoint of network reliability design, this asks which bipartite topology maximizes the number of connected spanning backbones under prescribed order and diameter constraints. Before presenting the main theorem, some preliminary concepts are required. Let $G$ be a graph with $V(G)=\{v_{1},v_{2},\ldots,v_{k}\}$, and let $H_1,H_2,\ldots,H_k$ be any $k$ graphs. The {\it generalized join graph} $G\circ(H_{1},H_{2},\ldots,H_{k})$ is obtained from $G$ by replacing each vertex~$v_{i}$ with a graph $H_{i}$ and joining each vertex in $H_{i}$ with each vertex in $H_{j}$ if $v_{i}v_{j}\in E(G)$. If each $H_{i}$ is an independent set of $n_{i}$ vertices, then we write $G\circ(H_{1},H_{2},\ldots,H_{k})$ as $G\circ(n_{1},n_{2},\ldots,n_{k})$, and call this graph a {\it blow-up} of $G$ with respect to the positive integer vector $(n_1,n_2,\ldots,n_k)$. Denote by  $K_{a,b}$  the complete bipartite graph with  bipartition $X\cup Y$, where $|X|=a$ and $|Y|=b$. For $d\ge4$ and $w=n-d-1\ge11$, the main difficulty is to show
that every extremal graph is a blow-up of a path whose integer vector has exactly three consecutive nontrivial parts, that is, parts of size greater than $1$.
Once this structure is established, the optimal sizes of these
parts can be determined by balancing arguments.
Our main theorems can now be stated as follows.

\begin{theorem}\label{main_th1}
If $G\in \mathcal{B}(n, d)$ with $d\in \{2,3\}$, then $G$ uniquely maximizes the number of spanning trees if and only if 
\begin{displaymath}
G\cong \begin{cases}
    K_{\lfloor \frac{n}{2} \rfloor, \lceil \frac{n}{2} \rceil}, &\mbox {\rm if $d=2$};\\
K_{\lfloor \frac{n}{2} \rfloor, \lceil \frac{n}{2} \rceil}', &\mbox {\rm if $d=3$},
   \end{cases}
\end{displaymath}
where $K_{\lfloor \frac{n}{2} \rfloor, \lceil \frac{n}{2} \rceil}'$ is the graph obtained from $K_{\lfloor \frac{n}{2} \rfloor, \lceil \frac{n}{2} \rceil}$ by deleting any edge.
\end{theorem}

\begin{theorem}\label{main_th2}
If $G\in \mathcal{B}(n,d)$ with $d\geq4$ and  $w=n-d-1\ge 11$, then
\begin{equation*}
\begin{split}
\tau(G)\leq
\left(\left\lfloor\frac{w+1}{4}\right\rfloor+1\right)
\left(\left\lceil\frac{w}{4}\right\rceil+1\right)
\left(\left\lfloor\frac{w}{2}\right\rfloor+1\right)
\left(\left\lfloor\frac{w}{2}\right\rfloor+2\right)^{
\left\lceil\frac{w}{2}\right\rceil}
\left(\left\lceil\frac{w}{2}\right\rceil+2\right)^{
\left\lfloor\frac{w}{2}\right\rfloor}.
\end{split}
\end{equation*}
The equality holds if and only if
\[
G\cong P_{d+1}\circ(n_1,n_2,\ldots,n_{d+1}),
\]
where there exists an index $i$ with $3\leq i\leq d-1$ such that $n_i=\left\lfloor\frac{w}{2}\right\rfloor+1$,
$\{n_{i-1},n_{i+1}\}=\left\{\left\lfloor\frac{w+1}{4}\right\rfloor+1,\left\lceil\frac{w}{4}\right\rceil+1\right\}$
and
$n_j=1$ for all $j \notin\{i-1,i,i+1\}$.
\end{theorem}

\section{Preliminaries}
We write $\mathbb{N}^+=\{1,2,3,\ldots\}$ and
$[k]=\{1,2,\ldots,k\}$ for every $k\in\mathbb{N}^+$. Suppose that $H$ is a graph with $V(H)=\{v_1,v_2,\ldots,v_n\}$. The {\it Laplacian matrix} $L_{H}$ of $H$ is the $n \times n$ symmetric matrix with entries
\begin{displaymath}
(L_{H})_{ij}=
   \begin{cases}
  d_H(v_i), &\mbox {\rm if $v_i=v_j$},\\
 -1, &\mbox {\rm if $v_iv_j\in E(H)$},\\
 0,  &\mbox {\rm otherwise}.
   \end{cases}
\end{displaymath}
Eigenvalues of $L_H$ are called the {\it Laplacian eigenvalues} of $H$. Let $\mu_1(H) \geq \cdots \geq \mu_{n-1}(H) \geq \mu_{n}(H)=0$ be all the Laplacian eigenvalues  of $H$.

Let $G$ be a graph with $k$ vertices, and suppose that $H_1, H_2, \ldots, H_k$ are graphs. Zhou and Bu \cite{Z1} gave a formula for counting spanning trees in $G\circ(H_{1},H_{2}, \ldots, H_{k})$.

\begin{theorem}[\hspace{1sp}{\cite{Z1}}]\label{theorem2.1}
Let $G_{0}=G\circ(H_{1},H_{2}, \ldots, H_{k})$ be a generalized join graph of a connected graph $G$ with $n_{i}=|V(H_{i})|$ for $i\in [k]$. Then
\begin{displaymath}
\tau(G_0)=\prod_{v_{i}\in V(G)}\frac{\prod_{\ell=1}^{n_{i}-1}\left(\mu_{\ell}(H_{i})+\sum_{v_{j}\in N_{G}(v_{i})}n_{j}\right)}{n_{i}}\sum_{T\in \mathcal{T}(G)}\prod_{v_{i}v_{j}\in E(T)}n_{i}n_{j}.
\end{displaymath}
\end{theorem}

Denote by $P_d$ the path of order $d$. Using Theorem \ref{theorem2.1}, we obtain the following formula for the number of spanning trees in the blow-up graph $P_{d+1}\circ(n_1,n_2, \ldots, n_{d+1})$.
\begin{corollary}\label{cor1}
For any positive integer vector $(n_1, n_2, \ldots, n_{d+1})$, we have
\begin{displaymath}
\tau\bigl(P_{d+1}\circ(n_1,n_2,\ldots,n_{d+1})\bigr) = n_{2}^{n_{1}-1} \cdot \left( \prod_{i=2}^{d} (n_{i-1}+n_{i+1})^{n_{i}-1} n_{i} \right) \cdot n_{d}^{n_{d+1}-1}.
\end{displaymath}
\end{corollary}
\begin{proof}
Let $G_0=P_{d+1}\circ(n_1,n_2,\ldots,n_{d+1})$. Note that $G_0=P_{d+1}\circ(H_1,H_2,\ldots,H_{d+1})$, where each $H_{i}$ is an independent set of $n_{i}$ vertices. Then $\mu_{\ell}(H_{i})=0$ for $1\leq\ell\leq n_i-1$ and $1\leq i\leq d+1$.  By Theorem \ref{theorem2.1}, we have
\begin{displaymath}
\begin{split}
\tau(G_0)&=\frac{(n_{2})^{n_{1}-1}(n_{d})^{n_{d+1}-1}}{n_1n_{d+1}}\prod_{i=2}^{d}\frac{ (n_{i-1}+n_{i+1})^{n_{i}-1}}{n_i} \sum_{T\in \mathcal{T}(P_{d+1})}\prod_{v_{i}v_{j}\in E(T)}n_{i}n_{j}\\
&=\frac{(n_{2})^{n_{1}-1}(n_{d})^{n_{d+1}-1}}{n_1n_{d+1}}\prod_{i=2}^{d}\frac{ (n_{i-1}+n_{i+1})^{n_{i}-1}}{n_i}n_1n_{d+1}\prod_{i=2}^{d} n_{i}^{2}\\
&=n_{2}^{n_{1}-1} \cdot \left( \prod_{i=2}^{d} (n_{i-1}+n_{i+1})^{n_{i}-1}  n_{i} \right) \cdot n_{d}^{n_{d+1}-1}.
\end{split}
\end{displaymath}
This completes the proof.
\end{proof}

 Denote by $\overline{G}$ the complement of  $G$. For a graph $G$ and any edge $e\in E(\overline{G})$, let $G+e$ be the graph obtained from $G$ by adding the edge $e$.   The following lemma is obvious.

\begin{lemma}\label{lem1}
Let $G$ be a noncomplete connected graph with $e\in E(\overline{G})$. Then we have $\tau(G) < \tau(G+e)$.
\end{lemma}

Let $G$ be a graph with vertex set $V(G)=\{v_{1},v_{2},\ldots,v_{k}\}$. Note that the blow-up graph $G_0=G\circ(n_1,n_2,\ldots,n_{k})$ has a vertex partition $V(G_0)=V_{1} \sqcup V_{2} \sqcup \cdots \sqcup V_{k}$,  where $|V_{i}|=n_{i}$ and $V_{i}=\{v_{i}^{1},v_{i}^{2},\ldots,v_{i}^{n_{i}}\}$ for $1\leq i\leq k$. In order to find the graph with the maximum number of spanning trees in blow-up graphs, we need the following result.

\begin{lemma}\label{lem2}
Let $G$ be a connected graph with $k$ vertices and $(n_{1},n_{2},\ldots,n_{k})$ be a positive integer vector. For  $i, j \in [k]$, if $v_{i}v_{j}\notin E(G)$, $N_{G}(v_{i})\subset N_{G}(v_{j})$ and $n_i\geq 2$, then
\begin{equation*}
\begin{split}
\tau\bigl(G\circ(n_{1},\ldots,n_{i},\ldots,n_{j},\ldots,n_{k})\bigr)
<\tau\bigl(G\circ(n_{1},\ldots,n_{i}-1,\ldots,n_{j}+1,\ldots,n_{k})\bigr).
\end{split}
\end{equation*}
\end{lemma}
\begin{proof}
Observe that
\begin{equation*}
\begin{split}
G\circ(n_{1},\ldots,n_{i},\ldots,n_{j},\ldots,n_{k})
={} & G\circ(n_{1},\ldots,n_{i}-1,\ldots,n_{j}+1,\ldots,n_{k})\\
&-\bigl\{v_{j}^{n_{j}+1}v_{\ell}^{r}:1\leq r\leq n_{\ell}\text{ and }v_{\ell}\in N_{G}(v_{j}) \setminus N_{G}(v_{i})\bigr\}.
\end{split}
\end{equation*}
Then the result follows from Lemma \ref{lem1}.
\end{proof}

Next we present several inequalities that will be used in the proofs of the main results.
\begin{lemma}\label{xsh}
The following inequalities hold.
\begin{enumerate}
\item[\normalfont(1)] If $m\in \mathbb{N}^+$ and $x\geq -1$, then
$(1+x)^m\geq 1+mx$.
\item[\normalfont(2)] If $0\le x<1$, then
$\ln(1-x)\leq -x$.
\item[\normalfont(3)] If $0\le x<1$, then
$\ln(1-x)\leq -x-\frac{x^2}{2}$.
\item[\normalfont(4)] If $x\geq0$, then
$\ln(1+x)\leq x-\frac{x^2}{2(1+x)}$.
\end{enumerate}
\end{lemma}
\begin{proof}
Inequality {\rm (1)} is Bernoulli's inequality, while the other inequalities can be verified directly by differentiation.
\end{proof}

\begin{lemma}\label{lem:T3lower}
 Let $\beta=\left\lfloor\frac w2\right\rfloor$, $\alpha=\left\lfloor\frac{w-\beta}{2}\right\rfloor$ and $\gamma=w-\alpha-\beta$, where $w\geq11$ is an integer. Define two functions
\begin{align*}
T_3(w) &\coloneqq
(1+\alpha)(1+\beta)(1+\gamma)
(\beta+2)^{\alpha+\gamma}
(\alpha+\gamma+2)^\beta;\\
g(w) &\coloneqq
(w+3)\ln\left(\frac w2+2\right)-\ln4-\frac14.
\end{align*}
Then $\ln T_3(w)>g(w)$.

\end{lemma}

\begin{proof}
If $w=2m$, then $m\ge6$, $\beta=m$, and $\alpha+\gamma=m$.  Therefore
\[
\frac{4T_3(w)}{(m+2)^{2m+3}}
=
\frac{4(\alpha+1)(m+1)(\gamma+1)}{(m+2)^3}.
\]
Because $\alpha,\gamma$ differ by at most $1$, we have
$4(\alpha+1)(\gamma+1)\ge(m+1)(m+3)$. Then
\[
\frac{4T_3(w)}{(m+2)^{2m+3}}
\ge
\frac{(m+1)^2(m+3)}{(m+2)^3}=
\left(1-\frac1{m+2}\right)
\left(1-\frac1{(m+2)^2}\right)\ge\frac78\cdot\frac{63}{64}
>\frac45.
\]

If $w=2m+1$, then $m\ge5$, $\beta=m$, and
$\alpha+\gamma=m+1$. Hence
\[
T_3(w)
=
(\alpha+1)(m+1)(\gamma+1)
(m+2)^{m+1}(m+3)^m.
\]
Since
$\frac w2+2=m+\frac52$,
and
$w+3=2m+4$,
we have
\begin{equation}\label{lem2.6-eq1}
\frac{4T_3(w)}
{(m+5/2)^{2m+4}}=
\frac{4(m+1)(\alpha+1)(\gamma+1)}
{(m+5/2)^3}
\left(\frac{m+2}{m+5/2}\right)^{m+1}
\left(\frac{m+3}{m+5/2}\right)^m.
\end{equation}
Because $\alpha,\gamma$ differ by at most $1$ and
$\alpha+\gamma=m+1$,
\begin{equation}\label{xx-1}
\begin{split}
\frac{4(m+1)(\alpha+1)(\gamma+1)}
{(m+5/2)^3}
&\ge
\frac{(m+1)(m+2)(m+4)}
{(m+5/2)^3}\\
&=
\left(1-\frac1{2m+5}\right)
\left(1-\frac9{(2m+5)^2}\right)\\
&\ge
\frac{14}{15}\cdot\frac{24}{25}
=\frac{112}{125}>\frac89.
\end{split}
\end{equation}
Moreover, since $m\ge5$, we have
$1-\frac1{2m+5}\ge\frac{14}{15}$ and
$\frac{m}{(2m+5)^2}\le\frac1{30}$. By  Lemma \ref{xsh} (1), 
\begin{equation}\label{xx-2}
\begin{split}
\left(\frac{m+2}{m+5/2}\right)^{m+1}
\left(\frac{m+3}{m+5/2}\right)^m&=
\left(1-\frac1{2m+5}\right)^{m+1}
\left(1+\frac1{2m+5}\right)^m\\
&=
\left(1-\frac1{2m+5}\right)
\left(1-\frac1{(2m+5)^2}\right)^m\\
& \ge\left(1-\frac1{2m+5}\right) \left(1-\frac{m}{(2m+5)^2}\right)\\
&\ge \frac{14}{15}\cdot \frac{29}{30}.
\end{split}
\end{equation}
By \eqref{lem2.6-eq1}, \eqref{xx-1} and \eqref{xx-2}, we obtain
\[
\frac{4T_3(w)}
{(m+5/2)^{2m+4}}
>
\frac89\cdot\frac{14}{15}\cdot\frac{29}{30}
=
\frac{1624}{2025}
>
\frac45.
\]

Thus, in both parities
\[
\ln T_3(w)-
\left((w+3)\ln\left(\frac w2+2\right)-\ln4\right)
>\ln\frac45
=-\ln\frac54
=-\ln\left(1+\frac14\right)
>-\frac14,
\]
which completes the proof.
\end{proof}

\begin{lemma}\label{lem:FQ}
Let $w\geq11$ and $0\leq Q\leq\frac{(w-1)^2}{4}$.
Then
\[
\begin{aligned}
3\ln\left(
\frac{w+\frac{w^2}{4}-Q+3}{3}
\right)
+w\ln\left(2+\frac{2Q}{w}\right)<
(w+3)\ln\left(\frac w2+2\right)-\ln4-\frac14.
\end{aligned}
\]
\end{lemma}

\begin{proof}
For $w\geq11$ and $0\leq Q\leq\frac{(w-1)^2}{4}$, let
\[
f(Q) \coloneqq
3\ln\left(
\frac{w+\frac{w^2}{4}-Q+3}{3}
\right)
+w\ln\left(2+\frac{2Q}{w}\right).
\]
A quick computation gives
\[
\begin{aligned}
f'(Q)=
-\frac{3}{w+\frac{w^2}{4}-Q+3}
+\frac{w}{w+Q}=
\frac{\frac{w^3}{4}+w^2-(w+3)Q}
{\left(w+\frac{w^2}{4}-Q+3\right)(w+Q)}.
\end{aligned}
\]
Since $0\leq Q\leq\frac{(w-1)^2}{4}$,
we have $w+Q>0$ and
\begin{align*}
w+\frac{w^2}{4}-Q+3 &\geq
w+\frac{w^2}{4}-\frac{(w-1)^2}{4}+3
=
\frac{6w+11}{4}>0;\\
\frac{w^3}{4}+w^2-(w+3)Q
& \geq
\frac{w^3}{4}+w^2
-\frac{(w+3)(w-1)^2}{4}=
\frac{3w^2+5w-3}{4}>0.
\end{align*}
Thus $f'(Q)>0$. Hence $f(Q)$ is strictly increasing on
$\left[0,\frac{(w-1)^2}{4}\right]$, and therefore
\begin{equation*}
f(Q)
\leq
f\left(\frac{(w-1)^2}{4}\right)=
3\ln\left(\frac{6w+11}{12}\right)
+w\ln\left(\frac{(w+1)^2}{2w}\right).
\end{equation*}
Since
\[
\frac{6w+11}{12}
=
\left(\frac w2+2\right)
\left(1-\frac{13}{6w+24}\right)\quad\text{and}\quad 
\frac{(w+1)^2}{2w}
=
\left(\frac w2+2\right)
\left(1-\frac{2w-1}{w(w+4)}\right),
\]
we obtain
\begin{equation}\label{lem2.7-eq1}
f(Q)
\leq
(w+3)\ln\left(\frac w2+2\right)
+3\ln\left(1-\frac{13}{6w+24}\right)
+w\ln\left(1-\frac{2w-1}{w(w+4)}\right).
\end{equation}
For $w\geq11$, we have
$0<\frac{13}{6w+24}<1$ and $0<\frac{2w-1}{w(w+4)}<1$.
From \eqref{lem2.7-eq1}, using Lemma \ref{xsh} (2) and
$\ln4<\frac32$, we obtain
\[
\begin{aligned}
f(Q)-
\left(
(w+3)\ln\left(\frac w2+2\right)-\ln4
\right)
&\leq
3\ln\left(1-\frac{13}{6w+24}\right)
+w\ln\left(1-\frac{2w-1}{w(w+4)}\right)
+\ln4\\
&<
-\frac{39}{6w+24}
-\frac{2w-1}{w+4}
+\frac32\\
&=
-\frac14-\frac{w-6}{4(w+4)}<
-\frac14,
\end{aligned}
\]
which implies the desired result.
\end{proof}

\section{Proof of Theorem \ref{main_th1}}
In this section we prove Theorem \ref{main_th1}. We treat the cases $d=2$ and $d=3$ separately in the following two propositions. These results   directly lead to Theorem \ref{main_th1}.

\begin{proposition}\label{d2_prop}
    Let $n \ge 3$. Then $K_{\lfloor \frac{n}{2} \rfloor, \lceil \frac{n}{2} \rceil}$ uniquely attains the maximum number of spanning trees in $\mathcal{B}(n, 2)$.
\end{proposition}
\begin{proof}
    Among the bipartite graphs of order $n$, those of diameter two are precisely the complete bipartite graphs. By observing that $K_{a, b} \cong P_2 \circ (a, b)$, Corollary \ref{cor1} gives $\tau(K_{a, b}) = a^{b - 1} b^{a - 1}$. Now, let $f \colon [0, \frac{n}{2} - 1] \to \mathbb{R}$ be defined by
    \[
        f(x) = \left(\frac{n}{2} - x - 1\right) \ln \left(\frac{n}{2} + x\right) + \left( \frac{n}{2} + x - 1 \right) \ln \left(\frac{n}{2} - x\right),
    \]
    so that $\ln \tau(K_{a, b}) = f\left(\frac{|a - b|}{2}\right)$ for any $a, b \in \mathbb{N}^+$ with $a + b = n$. A routine computation gives
    \[
        f'(x) = \ln\frac{\frac{n}{2} - x}{\frac{n}{2} + x} - \frac{2x(n - 1)}{\left(\frac{n}{2} + x\right)\left( \frac{n}{2} - x \right)} < 0
    \]
    for every $x \in (0, \frac{n}{2} - 1 )$, which means that $f$ is a decreasing function. Therefore, $K_{\lfloor \frac{n}{2} \rfloor, \lceil \frac{n}{2} \rceil}$ uniquely maximizes the number of spanning trees in $\mathcal{B}(n, 2)$.
\end{proof}

\begin{proposition}\label{d3_prop}
    Let $n \ge 4$. Then the graph obtained from $K_{\lfloor \frac{n}{2} \rfloor, \lceil \frac{n}{2} \rceil}$ by deleting any edge uniquely attains the maximum number of spanning trees in $\mathcal{B}(n, 3)$.
\end{proposition}
\begin{proof}
    Let $K'_{a, b}$ be the graph obtained from $K_{a, b}$ by deleting any edge. First, we observe that any extremal graph must be of the form $K'_{a, b}$ with $a + b = n$ and $a, b \ge 2$. Indeed, suppose that $G^*$ maximizes the number of spanning trees in $\mathcal{B}(n, 3)$, and let $a$ and $b$ be the sizes of its bipartition sets. Then $G^* \not\cong K_{a, b}$, so either $G^* \cong K'_{a, b}$, or $G^*$ is a proper subgraph of $K'_{a, b}$. Either way, since $G^*$ is connected, we have $a, b \ge 2$. Now, because $K'_{a, b} \in \mathcal{B}(n, 3)$, if $G^*$ is a proper subgraph of $K'_{a, b}$, Lemma \ref{lem1} would yield a contradiction to the maximality of $G^*$. Therefore, when maximizing the number of spanning trees in $\mathcal{B}(n, 3)$, it suffices to consider the graphs of the form $K'_{a, b}$ with $a + b = n$ and $a, b \ge 2$.

    Since $K'_{a, b} \cong P_4 \circ (1, a - 1, b - 1, 1)$, Corollary \ref{cor1} gives $\tau(K'_{a, b}) = (a - 1)(b - 1) a^{b - 2} b^{a - 2}$. Now, let $f \colon [0, \frac{n}{2} - 2] \to \mathbb{R}$ be defined by
    \begin{align*}
        f(x) &= \left( \frac{n}{2} - x - 2 \right) \ln \left( \frac{n}{2} + x \right) + \left( \frac{n}{2} + x - 2 \right) \ln \left( \frac{n}{2} - x \right)\\
        &\phantom{\coloneqq}\, + \ln \left( \frac{n}{2} + x - 1 \right) + \ln \left( \frac{n}{2} - x - 1 \right) ,
    \end{align*}
    so that $\ln \tau(K'_{a, b}) = f\left( \frac{|a - b|}{2} \right)$ for any $a, b \ge 2$ with $a + b = n$. A routine computation gives
    \[
        f'(x) = \ln \frac{\frac{n}{2} - x}{\frac{n}{2} + x} - \frac{2x(n - 2)}{\left(\frac{n}{2} + x\right)\left( \frac{n}{2} - x \right)} - \frac{2x}{\left(\frac{n}{2} + x - 1\right)\left( \frac{n}{2} - x - 1 \right)} < 0
    \]
    for every $x \in (0, \frac{n}{2} - 2)$, hence $f$ is a decreasing function. Therefore, $K'_{\lfloor \frac{n}{2} \rfloor, \lceil \frac{n}{2} \rceil}$ uniquely maximizes the number of spanning trees in $\mathcal{B}(n, 3)$.
\end{proof}

\section{Proof of Theorem \ref{main_th2}}\label{sc_main}
In the present section we prove Theorem \ref{main_th2}. Put $w=n-d-1\ge 11$.  For $d\geq4$, let $G^{*}\in \mathcal{B}(n,d)$ be a graph that attains the maximum number of spanning trees. 
Choose two vertices $v_1$ and $v_{d+1}$ of $G^*$ such that $d_{G^*}(v_1,v_{d+1})=d$. For $i\in [d+1]$, let
$V_i=\{u\in V(G^*): d_{G^*}(v_1,u)=i-1\}$.
Then
\[
V(G^*)=V_1\sqcup V_2\sqcup\cdots\sqcup V_{d+1},
\]
where $V_1=\{v_1\}$ and $V_{i}\neq\emptyset$ for $i=2,3,\ldots,d+1$. We refer to each $V_i$ as a partition subset of $G^*$. Keeping  the above notations, we will prove some statements on the structure of $G^*$, which will eventually show the main result.

\begin{lemma}\label{xh1}
$V_i$ is  an independent set for each $i\in[d+1]$.
\end{lemma}
\begin{proof}
Suppose that  there exists an integer $k\in [d+1]\setminus \{1\}$ satisfying $x,y\in V_k$ and $xy\in E(G^*)$. Then $d_{G^*}(v_1,x)=d_{G^*}(v_1,y)=k-1$, so we can find a closed walk of length $2k-1$, which contradicts the bipartiteness of $G^*$.
\end{proof}

\begin{lemma}\label{xh2}
$G^*[V_{i} \cup V_{i+1}]$ is a complete bipartite subgraph for each $i \in [d]$.
\end{lemma}
\begin{proof}
Indeed, if there exist nonadjacent vertices $x\in V_i$ and $y\in V_{i+1}$ for some $i\in [d]$, then adding the edge $xy$ to $G^*$ preserves the bipartiteness and does not change the order or diameter of the graph. By Lemma \ref{lem1}, this increases the number of spanning trees, contradicting the maximality of $G^*$. Then the result follows from Lemma \ref{xh1}.
\end{proof}

By Lemma \ref{xh2}, we have $G^*\cong P_{d+1}\circ(n_1,n_2, \ldots, n_{d+1})$. Let $V(P_{d+1})=\{v_{1},v_{2}, \ldots, v_{d+1}\}$. According to the partition of $G^*$, it is known that $n_1=1$. Note that  $N_{P_{d+1}}(v_{d+1})\subset N_{P_{d+1}}(v_{d-1})$. By Lemma \ref{lem2}, we infer that $n_{d+1}=1$. Define the $(d+1)$-dimensional integer vector $\mathbf{n}=(1,n_2,n_3, \ldots, n_{d},1)$. Thus,
$G^*\cong P_{d+1}\circ(1,n_2,n_3, \ldots, n_d,1)=P_{d+1}\circ\mathbf{n}.$
By Corollary \ref{cor1}, we have
\begin{equation}\label{eq1}
\begin{split}
\tau(G^*)=\prod_{i=2}^{d}n_i (n_{i-1}+n_{i+1})^{n_{i}-1},
\end{split}
\end{equation}
where $n_1=n_{d+1}=1$.
We continue with proving some properties of $n_2,n_3,\dots,n_d$, which will further restrict the class of candidate extremal graphs.

\begin{lemma}\label{xh3}
The entries greater than $1$ among
$n_2,n_3,\ldots,n_d$ occur consecutively.
\end{lemma}
\begin{proof}
Suppose, to the contrary, that the entries greater than $1$ among
$n_2,n_3,\ldots,n_d$ do not occur consecutively. Then there exist
$k_1,k_2\in\{3,4,\ldots,d-1\}$ such that $k_1\leq k_2$, $n_i=1$ for all $k_1\leq i\leq k_2$,
 $n_{k_1-1}\geq2$ and
$n_{k_2+1}\geq2$.
Now, let $G'$ be the graph obtained from $G^*$ by permuting the blow-up integer vector so that the entries $n_i=1$ with $k_1\leq i\leq k_2$ are placed at the beginning, i.e., let
\[
G'=P_{d+1}\circ
\left(
\underbrace{1,1,\ldots,1}_{\mbox{\scriptsize $k_2-k_1+2$ times}},
n_2,n_3,\ldots,n_{k_1-1},
n_{k_2+1},n_{k_2+2},\ldots,n_d,1
\right).
\]
Clearly, $G'\in\mathcal{B}(n,d)$. By Corollary \ref{cor1} and
\eqref{eq1}, we have
\[
\begin{aligned}
\frac{\tau(G')}{\tau(G^*)}
=
\frac{
\left(n_{k_1-2}+n_{k_2+1}\right)^{n_{k_1-1}-1}
\left(n_{k_1-1}+n_{k_2+2}\right)^{n_{k_2+1}-1}
}{
\left(n_{k_1-2}+1\right)^{n_{k_1-1}-1}
\left(1+n_{k_2+2}\right)^{n_{k_2+1}-1}
}>1,
\end{aligned}
\]
where the strict inequality follows from
$n_{k_1-1},n_{k_2+1}\geq2$. Hence
$\tau(G')>\tau(G^*)$, contradicting the maximality of $G^*$. This completes the proof.
\end{proof}

 Let $S=\{i\in \{2,3,\ldots,d\}:n_i\geq 2\}$. By Lemma \ref{xh3}, we know that $S$ is a set of consecutive integers. Suppose that $S=\{s,s+1,\ldots,s+r-1\}$
where  $2\leq s\leq s+r-1\leq d$ and $1\leq r=|S|\leq d-1$, that is, $r$ is the number of $n_i\geq 2$ with $i\in \{2,3,\ldots,d\}$ in $G^*\cong P_{d+1}\circ(1,n_2,n_3, \ldots, n_d,1)$. For convenience, write 
$x_j=n_{s+j-1}-1$ for $1\leq j\leq r$,
and set $x_0=x_{r+1}=0$. Then $x_j\geq1$ for every $j\in[r]$. Since $n_i=1$ for every $i \in [d + 1] \setminus S$, 
\[
\sum_{i=1}^{r}x_i=\sum_{i=1}^{d+1}(n_i-1)=n-d-1=w\geq11.
\]
 By \eqref{eq1}, the spanning tree formula  of $G^*$ becomes
\begin{equation}\label{eq2}
\tau(G^*)=
\prod_{i=1}^{r}
(1+x_i)(2+x_{i-1}+x_{i+1})^{x_i}.
\end{equation}

For each $w\in\mathbb{N}^+$, let $\mathcal{A}_w$ denote the set of all finite
sequences of positive integers $(a_1,a_2,\ldots,a_k)$ such that $\sum_{i=1}^{k}a_i=w$.
Furthermore, for any $q\in\mathbb{N}^+$, let $\mathcal{A}_w^q$ denote the
subset of $\mathcal{A}_w$ consisting of the sequences with at most $q$
entries. For any $(a_1,a_2,\ldots,a_k)\in\mathcal{A}_w$, let $\varphi$ be the function defined by
\[
\varphi(a_1,a_2,\ldots,a_k)
\coloneqq
\prod_{i=1}^{k}
(1+a_i)(2+a_{i-1}+a_{i+1})^{a_i},
\]
where $a_0=a_{k+1}=0$. By \eqref{eq2}, we observe that $\tau(G^*)=\varphi(x_1,x_2,\ldots,x_r)$.

Notice that every sequence $(a_1,a_2,\ldots,a_k)\in\mathcal{A}_w^{d-1}$ determines a graph in $\mathcal{B}(n,d)$ by assigning the values
$1+a_1,1+a_2,\ldots,1+a_k$
to $k$ consecutive internal entries of the integer vector corresponding to  $P_{d+1}$ and assigning $1$ to all the other parts. By  \eqref{eq1}, the number of spanning trees of this graph is $\varphi(a_1,a_2,\ldots,a_k)$.
Since $G^*$  attains the maximum number of spanning trees among all graphs in $\mathcal{B}(n,d)$, the sequence
$(x_1,x_2,\ldots,x_r)\in\mathcal{A}_w^{d-1}$
maximizes $\varphi$ over $\mathcal{A}_w^{d-1}$ with $d\ge 4$. Therefore, in order to
prove that $r=3$, it suffices to show that every sequence maximizing
$\varphi$ over $\mathcal{A}_w^{d-1}$ with $d\ge4$ has exactly three entries.  We first exclude the cases $r=1$ and $2$. 

\begin{lemma}\label{xh4}
$r\neq 1,2$.
\end{lemma}
\begin{proof}
We first show that $r\neq 1$. Suppose, to the contrary, that $r=1$. Then $x_1=w=n-d-1$. 
Since $(w-1,1)\in\mathcal{A}_w^2\subseteq\mathcal{A}_w^{d-1}$, we have 
\[ 
\begin{split} 
\frac{\varphi(w-1,1)}{\varphi(w)} 
&= 
w\left(\frac{3}{2}\right)^{w-1} 
>1. 
\end{split} 
\] 
Thus $\varphi(w-1,1)>\varphi(w)$. This contradicts the fact that $(x_1)=(w)$ maximizes $\varphi$ over
$\mathcal{A}_w^{d-1}$. Therefore $r\neq1$.

Next we show that $r\neq2$. Suppose that $r=2$ and write $(x_1,x_2)=(a,b)$.
By reversing the sequence if necessary, assume that $b\ge a$.
Since $a+b=w\ge11$, we have $b\ge2$, and hence
$(1,a,b-1)\in\mathcal A_w^3\subseteq\mathcal A_w^{d-1}$.
However,
$$
\frac{\varphi(1,a,b-1)}{\varphi(a,b)}
=\frac{2b}{b+1}>1,
$$
contradicting the maximality of $(a,b)$.
Therefore $r\neq2$.
\end{proof}

Now we exclude the case $r=4$.
\begin{lemma}\label{xh5}
$r\neq 4$.
\end{lemma}
\begin{proof}
Suppose that $r=4$. Write $(x_1,x_2,x_3,x_4)=(a,b,c,h)$, where
$a,b,c,h\in\mathbb{N}^+$  and
$a+b+c+h=w$.
By the definition of $\varphi$, we have
\[
\begin{aligned}
\varphi(a,b,c,h)
= (a+1)(b+1)(c+1)(h+1)
(b+2)^a(a+c+2)^b(b+h+2)^c(c+2)^h.
\end{aligned}
\]

By reversing the order if necessary, we may assume that $a+c\geq b+h$. Since $a+b+c+h=w\geq11$, it follows that
$a+c\geq6$. Choose $p,q \in \mathbb{N}^+$ such that
$p+q=a+c$ and $|p-q|\leq1$.
Then $(p,b+h,q)\in\mathcal{A}_w^3
\subseteq\mathcal{A}_w^{d-1}$. 
By the definition of $\varphi$,
\[
\varphi(p,b+h,q)
=
(p+1)(b+h+1)(q+1)
(b+h+2)^{a+c}(a+c+2)^{b+h}.
\]
Since $p$ and $q$ are as equal as possible,
$(p+1)(q+1)
=
\left\lfloor\frac{(a+c+2)^2}{4}\right\rfloor$.
Hence
\begin{equation}\label{eq3}
\begin{aligned}
\frac{\varphi(a,b,c,h)}
{\varphi(p,b+h,q)}
=
\frac{(a+1)(b+1)(c+1)(h+1)}
{(b+h+1)\left\lfloor (a+c+2)^2/4\right\rfloor}
\left(\frac{b+2}{b+h+2}\right)^a
\left(\frac{c+2}{a+c+2}\right)^h.
\end{aligned}
\end{equation}
For every integer $a+c$, we have
$\left\lfloor\frac{(a+c+2)^2}{4}\right\rfloor
\geq
\frac{(a+c+1)(a+c+3)}{4}$.
Therefore
\begin{equation}\label{eq4}
\begin{aligned}
\frac{\varphi(a,b,c,h)}
{\varphi(p,b+h,q)}
\leq
\frac{4(a+1)(b+1)(c+1)(h+1)}
{(b+h+1)(a+c+1)(a+c+3)}
\left(\frac{b+2}{b+h+2}\right)^a
\left(\frac{c+2}{a+c+2}\right)^h.
\end{aligned}
\end{equation}
Put $m=a+c+1$. Then $m\ge7$ and $c=m-1-a$.
For fixed $a,c,h$, the right-hand side of \eqref{eq4}
is strictly increasing in $b$, since both
$(b+1)/(b+h+1)$ and $(b+2)/(b+h+2)$
are strictly increasing in $b$.
Since $b+h\le a+c$, we have $b\le m-1-h$.
We therefore bound the right-hand side of \eqref{eq4}
by replacing $b$ with $m-1-h$, obtaining
\begin{equation}\label{eq5}
\frac{\varphi(a,b,c,h)}
{\varphi(p,b+h,q)}
\leq
\frac{4(a+1)(h+1)(m-a)(m-h)}
{m^2(m+2)}
\left(1-\frac{h}{m+1}\right)^a
\left(1-\frac{a}{m+1}\right)^h.
\end{equation}

We now estimate the right-hand side of \eqref{eq5}. Let $\rho=ah$.
By Lemma \ref{xsh} (3),
\[
\begin{aligned}
a\ln\left(1-\frac{h}{m+1}\right)
+h\ln\left(1-\frac{a}{m+1}\right)
\leq
-\frac{2\rho}{m+1}
-\frac{\rho(a+h)}{2(m+1)^2}\leq
-\frac{2\rho}{m+1}
-\frac{\rho^{3/2}}{(m+1)^2},
\end{aligned}
\]
where the last inequality follows from $a+h\geq2\sqrt{ah}=2\sqrt{\rho}$.
Consequently,
\begin{equation}\label{eq6}
\frac{\varphi(a,b,c,h)}
{\varphi(p,b+h,q)}
\leq
\frac{4(a+1)(h+1)(m-a)(m-h)}
{m^2(m+2)}
\exp\left(
-\frac{2\rho}{m+1}
-\frac{\rho^{3/2}}{(m+1)^2}
\right).
\end{equation}
Moreover, by $4xy\leq(x+y)^2$, we have
\[
\begin{aligned}
4m(a+1)(h+1)(m-a)(m-h)
&\leq
\left(m(a+1)(h+1)+(m-a)(m-h)\right)^2\\
&=
(m+1)^2(m+\rho)^2.
\end{aligned}
\]
It follows from \eqref{eq6} that
\begin{equation}\label{eq7}
\frac{\varphi(a,b,c,h)}
{\varphi(p,b+h,q)}
\leq \Psi_m(\rho),
\end{equation}
where
\[
\Psi_m(\rho) \coloneqq
\frac{(m+1)^2(m+\rho)^2}{m^3(m+2)}
\exp\left(
-\frac{2\rho}{m+1}
-\frac{\rho^{3/2}}{(m+1)^2}
\right).
\]

Suppose $\rho\geq2$. Regarding $\Psi_m(\rho)$ as a function of
the real variable $\rho\geq2$, we have
\[
\frac{\mathrm{d}}{\mathrm{d}\rho}\ln \Psi_m(\rho)
=
\frac{2}{m+\rho}
-\frac{2}{m+1}
-\frac{3\sqrt{\rho}}{2(m+1)^2}
<0.
\]
Thus $\Psi_m(\rho)$ is strictly decreasing for $\rho\geq2$, and hence 
\[\Psi_m(\rho)\leq
\Psi_m(2)
=
\frac{(m+1)^2(m+2)}{m^3}
\exp\left(
-\frac4{m+1}
-\frac{2\sqrt2}{(m+1)^2}
\right).
\]
By Lemma \ref{xsh} (4),
we obtain
\begin{equation*}
 \begin{split}
\ln \Psi_m(2)={}&
2\ln\left(1+\frac1m\right)
+\ln\left(1+\frac2m\right)
-\frac4{m+1}
-\frac{2\sqrt2}{(m+1)^2}\\
\leq{}&
\frac2m-\frac1{m(m+1)}
+\frac2m-\frac2{m(m+2)}
-\frac4{m+1}
-\frac{2\sqrt2}{(m+1)^2}\\
={}&
\frac{m+4}{m(m+1)(m+2)}
-\frac{2\sqrt2}{(m+1)^2}\\
<{}&\frac2{(m+1)^2}-\frac{2\sqrt2}{(m+1)^2}<{}0 ,
 \end{split}
\end{equation*}
where the second inequality holds as $m\geq7$. Hence
$\Psi_m(\rho)<1$ for $\rho\geq2$. By \eqref{eq7},
\[
\frac{\varphi(a,b,c,h)}
{\varphi(p,b+h,q)}
<1.
\]

It remains to consider $\rho=1$. Since $\rho=ah$ and $a, h \in \mathbb{N}^+$, $a = h = 1$. Then \eqref{eq5} gives
\[
\frac{\varphi(a,b,c,h)}
{\varphi(p,b+h,q)}
\leq
\frac{16(m-1)^2}{(m+1)^2(m+2)}
\leq1,
\]
because $(m+1)^2(m+2)-16(m-1)^2=(m-7)(m^2-5m+2)\geq0$ with $m\geq7$.
The inequality is strict if $m>7$. Finally, suppose that $m=7$. Then $a+c=6$, and the floor estimate used
in passing from \eqref{eq3} to \eqref{eq4} is strict:
\[
\frac{(a+c+1)(a+c+3)}{4}
=
\frac{63}{4}
<
16
=
\left\lfloor\frac{(a+c+2)^2}{4}\right\rfloor.
\]
Hence
\[
\frac{\varphi(a,b,c,h)}
{\varphi(p,b+h,q)}
\leq
\frac{63}{64}
<1.
\]

Thus $\varphi(a,b,c,h)<\varphi(p,b+h,q)$. Since $(p,b+h,q)\in\mathcal{A}_w^3\subseteq\mathcal{A}_w^{d-1}$,
this contradicts the fact that
$(x_1,x_2,x_3,x_4)=(a,b,c,h)$ maximizes $\varphi$ over
$\mathcal{A}_w^{d-1}$. Therefore $r\neq4$.
\end{proof}

For any $(a_1,a_2,\ldots,a_r)\in\mathcal{A}_w$, let  
$a_0=a_{r+1}=0$ and   $Q=\sum_{i=1}^{r-1}a_i a_{i+1}$. Taking logarithms in the definition of $\varphi$   
and applying
Jensen's inequality  with weights $a_i/w$ gives   
\begin{equation}\label{eq8}   
\begin{split}   
\ln\varphi(a_1,a_2,\ldots,a_r)   
&=\sum_{i=1}^{r}\ln(1+a_i)   
+\sum_{i=1}^{r}a_i\ln(2+a_{i-1}+a_{i+1})\\   
&\le   
\sum_{i=1}^{r}\ln(1+a_i)   
+w\ln\left(   
\frac1w\sum_{i=1}^{r}   
 a_i(2+a_{i-1}+a_{i+1})   
\right)\\   
&=   
\sum_{i=1}^{r}\ln(1+a_i)   
+w\ln\left(   
2+\frac{2}{w} Q  
\right).   
\end{split}   
\end{equation}   
Since $a_i\geq1$ and $\sum_{i=1}^{r}a_i=w$, we have

\[
\begin{aligned}
\sum_{i=1}^{r-1}(a_i-1)(a_{i+1}-1)
&\le
\left(\sum_{i\ \mathrm{odd}}(a_i-1)\right)
\left(\sum_{j\ \mathrm{even}}(a_j-1)\right)\\
&\le
\frac14\left(\sum_{i=1}^{r}(a_i-1)\right)^2
=
\frac{(w-r)^2}{4}.
\end{aligned}
\]
Then
\begin{equation}\label{eq9}
\begin{aligned}
Q
&=
\sum_{i=1}^{r-1}
\left[
1+(a_i-1)+(a_{i+1}-1)
+(a_i-1)(a_{i+1}-1)
\right]\\
&\le
(r-1)+2(w-r)+\frac{(w-r)^2}{4}.
\end{aligned}
\end{equation}
Let $A=\sum_{i\ {\rm odd}}a_i$ and
$B=\sum_{j\ {\rm even}}a_j$. Then
\begin{equation}\label{eq10}   
\begin{split}   
\sum_{\substack{i\ {\rm odd},\ j\ {\rm even}\\|i-j|\ne1}}
a_i a_j&=AB-Q=\frac14\left(   
A+ B 
\right)^2-\frac14\left(   
A-B  
\right)^2-Q\leq\frac{w^{2}}{4}-Q.   
\end{split}   
\end{equation}

We now exclude the remaining cases $r\ge5$.
Throughout the following lemmas, let
$\alpha,~\beta,~\gamma,~T_3(w)$ and $g(w)$
be as defined in Lemma~\ref{lem:T3lower}.
Since $d\ge4$, we have
$(\alpha,\beta,\gamma)\in\mathcal A_w^3
\subseteq\mathcal A_w^{d-1}$. Also, by the definition of $\varphi$, $T_3(w)=\varphi(\alpha,\beta,\gamma)$.

\begin{lemma}\label{xh6}
$r \neq 5,6$.
\end{lemma}
\begin{proof}
Assume that $r\in\{5,6\}$ and write
$(x_1,x_2,\ldots,x_r)=(a_1,a_2,\ldots,a_r)$.
By \eqref{eq9}, 
\[
Q\le
\begin{cases}
\dfrac{(w-1)^2}{4}, & r=5;\\[2mm]
\dfrac{w^2}{4}-w+2\le\dfrac{(w-1)^2}{4}, & r=6.
\end{cases}
\]
Thus, in both cases,
$Q\le\frac{(w-1)^2}{4}$.

We next estimate $\prod_{i=1}^{r}(1+a_i)$. If $r=5$, set
$U_1=1+a_3$,
$U_2=(1+a_1)(1+a_4)$ and
$U_3=(1+a_2)(1+a_5)$.
By \eqref{eq10}, we have
\[
U_1+U_2+U_3
=
w+3+a_1a_4+a_2a_5
\le
w+\frac{w^2}{4}-Q+3.
\]
If $r=6$, set
$U_1=(1+a_1)(1+a_4)$,
$U_2=(1+a_3)(1+a_6)$ and
$U_3=(1+a_5)(1+a_2)$.
By \eqref{eq10}, we have 
\[
U_1+U_2+U_3=
w+3+a_1a_4+a_3a_6+a_5a_2
\le
w+\frac{w^2}{4}-Q+3.
\]
In either case,
$U_1U_2U_3=\prod_{i=1}^{r}(1+a_i)$. Then the AM--GM inequality implies that

\begin{equation}\label{eq12}
\prod_{i=1}^{r}(1+a_i)
\le\left(
\frac{U_1+U_2+U_3}{3}
\right)^3\le
\left(
\frac{w+\frac{w^2}{4}-Q+3}{3}
\right)^3.
\end{equation}

Therefore, by \eqref{eq8} and \eqref{eq12}, we have
\begin{equation}\label{eq13}
\ln\varphi(a_1,a_2,\ldots,a_r)
\le
3\ln\left(
\frac{
w+\frac{w^2}{4}-Q+3
}{3}
\right)
+
w\ln\left(
2+\frac{2Q}{w}
\right).
\end{equation}
By  Lemma \ref{lem:FQ},
$\ln\varphi(a_1,a_2,\ldots,a_r)
<g(w)$.
By Lemma \ref{lem:T3lower},
\[
\ln\varphi(a_1,a_2,\ldots,a_r)
<g(w)<\ln T_3(w)=\ln\varphi(\alpha,\beta,\gamma),
\]
which contradicts the fact that
$(x_1,x_2,\ldots,x_r)
=(a_1,a_2,\ldots,a_r)$
maximizes $\varphi$ over $\mathcal{A}_w^{d-1}$.
Hence $r\neq5,6$.
\end{proof}

\begin{lemma}\label{xh7}
$r\notin\{7,8,\ldots,d-1\}$ with $r\le w/4$.
\end{lemma}
\begin{proof}
Suppose, to the contrary, that
$r\in\{7,8,\ldots,d-1\}$ with $r\le w/4$, and write
$(x_1,x_2,\ldots,x_r)=(a_1,a_2,\ldots,a_r)$. By \eqref{eq9},
\begin{equation}\label{eq14}
\frac{w^2}{4}-Q\ge
\Delta_r(w) \coloneqq
\frac{r-4}{2}w-\frac{r^2}{4}+r+1.
\end{equation}
Since $a_i\ge1$ and $\sum_{i=1}^r a_i=w$, we have  $r\le w$. Recall that $A=\sum_{i\ {\rm odd}}a_i$ and
$B=\sum_{j\ {\rm even}}a_j$, and write
\[
N=
\sum_{\substack{i\ {\rm odd},\,j\ {\rm even}\\|i-j|\ne1}}
a_i a_j.
\]
By \eqref{eq10}, we have
$N\le\frac{w^2}{4}-Q$. In order to show that $\ln\varphi(a_1,\ldots,a_r)<g(w)<\ln T_3(w)$, we need the following claim.

\vskip 0.25cm
    \begin{claim}\label{xshxsh-1}
        For $r\in \{7,8,\ldots,d-1\}$,
\begin{equation*}
\prod_{i=1}^{r}(1+a_i)
\le
\left(
\frac{w+\frac{w^2}{4}-Q+3}{3}
\right)^3
\left(
1+\frac{4(r-5)\left(\frac{w^2}{4}-Q\right)}
{(r-3)(r-6)w}
\right)^{r-6}.
\end{equation*}
\end{claim}
\begin{proof}[\em\textbf {Proof of Claim \ref{xshxsh-1}}]Choose the three largest $a_i$ with odd indices and the three largest
$a_i$ with even indices. Let $R_{\rm odd}$ and $R_{\rm even}$ be the
sums of the remaining entries with odd and even indices, respectively,
and put $R=R_{\rm odd}+R_{\rm even}$.

Let $a_{i_1}\ge a_{i_2}\ge a_{i_3}$ be the three selected entries with odd
indices. For every even index $j$, at most two odd indices $i$ satisfy
$|i-j|=1$. Hence
\begin{equation}\label{x-1}
\sum_{\substack{i\ {\rm odd}\\|i-j|\ne1}}a_{i}
\ge
A-a_{i_1}-a_{i_2}
=
a_{i_3}+R_{\rm odd}.
\end{equation}
Since every remaining entry with an odd index is at most $a_{i_3}$, and
there are $\lceil r/2\rceil-3$ such entries, we have
$
R_{\rm odd}
\le
\left(\left\lceil\frac r2\right\rceil-3\right)a_{i_3}$. It follows from \eqref{x-1} that
\[
\sum_{\substack{i\ {\rm odd}\\|i-j|\ne1}}a_i
\ge
\frac{\left\lceil r/2\right\rceil-2}
{\left\lceil r/2\right\rceil-3}R_{\rm odd}.
\]
Multiplying this inequality by $a_j$ and summing over all even indices
$j$, we obtain
\[
\begin{aligned}
N
=
\sum_{j\ {\rm even}}
a_j
\sum_{\substack{i\ {\rm odd}\\|i-j|\ne1}}a_i\ge
B\frac{\left\lceil r/2\right\rceil-2}
{\left\lceil r/2\right\rceil-3}R_{\rm odd}.
\end{aligned}
\]
Hence
\[
R_{\rm odd}
\le
\frac{\left\lceil r/2\right\rceil-3}
{\left\lceil r/2\right\rceil-2}\frac{N}{B}
\le
\frac{r-5}{r-3}\frac{N}{B}.
\]
Similarly, interchanging odd and even indices, we obtain $R_{\rm even}
\le
\frac{r-5}{r-3}\frac{N}{A}$, where $R_{\rm even}=0$ when $r=7$. Therefore
\begin{equation}\label{eq15}
\begin{aligned}
R\le
\frac{r-5}{r-3}
N\left(\frac1A+\frac1B\right)=
\frac{r-5}{r-3}\frac{Nw}{AB}\le
\frac{4(r-5)}{(r-3)w}
\left(\frac{w^2}{4}-Q\right) .
\end{aligned}
\end{equation}
The last inequality follows from 
\begin{align*}
4\left(\frac{w^2}{4}-Q\right)AB-Nw^2 &=
4\left(\frac{w^2}{4}-Q\right)AB-(AB-Q)(A+B)^2\\
&= \left((A + B)^2 - 4Q\right)AB-(AB-Q)(A+B)^2\\
&= Q(A-B)^2
\ge0.
\end{align*}

Let $j_1,j_2,j_3$ be the indices corresponding to the three selected
entries with even indices. Consider the bipartite graph with vertex
classes $\{i_1,i_2,i_3\}$ and
$\{j_1,j_2,j_3\}$, where $i_t$ is joined to $j_s$ whenever
$|i_t-j_s|\ne1$. We claim that this bipartite graph has a perfect matching. By Hall's theorem, it suffices to verify that every
subset of $\{i_1,i_2,i_3\}$ has at least as many neighbors in
$\{j_1,j_2,j_3\}$ as its cardinality.

Each selected odd index has at least one neighbor, since for a fixed
$i_t$ there are at most two even indices $j_s$ satisfying
$|i_t-j_s|=1$. Any two selected odd indices have at least two
neighbors; otherwise, two selected even indices would each satisfy
$|i_t-j_s|=1$ for both odd indices, which is impossible since two
distinct odd indices have at most one common even index at distance
one. Finally, the three selected odd indices have all three selected
even indices as neighbors, since a fixed even index can be at distance
one from at most two odd indices. Thus Hall's condition holds. Hence
there is a permutation $\sigma$ of $\{1,2,3\}$ such that
\[
|i_t-j_{\sigma(t)}|\ne1,
\qquad t=1,2,3.
\]
Therefore, by the definition of $N$ and
$N\le\frac{w^2}{4}-Q$,
\[
\sum_{t=1}^3 a_{i_t}a_{j_{\sigma(t)}}
\le
N
\le
\frac{w^2}{4}-Q.
\]
Since the six selected entries have total sum $w-R$,
\[
\begin{aligned}
\sum_{t=1}^3
(1+a_{i_t})(1+a_{j_{\sigma(t)}})
&=
3+\sum_{t=1}^3
(a_{i_t}+a_{j_{\sigma(t)}})
+\sum_{t=1}^3
a_{i_t}a_{j_{\sigma(t)}}\\
&\le
w-R+\frac{w^2}{4}-Q+3.
\end{aligned}
\]

Applying AM--GM to these three products and to the remaining $r-6$
terms, and then using \eqref{eq15}, gives
\begin{equation*}
\begin{aligned}
\prod_{i=1}^{r}(1+a_i)
&\le
\left(
\frac{w-R+\frac{w^2}{4}-Q+3}{3}
\right)^3
\left(1+\frac{R}{r-6}\right)^{r-6}\\
&\le
\left(
\frac{w+\frac{w^2}{4}-Q+3}{3}
\right)^3
\left(
1+\frac{4(r-5)\left(\frac{w^2}{4}-Q\right)}
{(r-3)(r-6)w}
\right)^{r-6}.
\end{aligned}
\end{equation*}
This completes the proof of Claim \ref{xshxsh-1}.
\renewcommand{\qedsymbol}{$\lrcorner$}
\end{proof}

Let
$
H_{w,r}(z) \coloneqq 
3\ln\left(\frac{w+z+3}{3}\right)
+(r-6)\ln\left(
1+\frac{4(r-5)z}{(r-3)(r-6)w}
\right)+
w\ln\left(
\frac w2+2-\frac{2z}{w}
\right)
$. Next we  show that  $H_{w,r}(z)$ is strictly decreasing on
$\left[
\Delta_r(w),\frac{w^2}{4}-Q
\right]$. Since $w\ge4r$ and $r\ge7$,
\begin{align*}
    \Delta_r(w)-w &\ge \frac{7r^2-44r+4}{4}>0,\\
    \Delta_r(w)-\frac{r-6}{4}w &\ge \frac{3r^2-4r+4}{4}>0.
\end{align*}
For every $z\in \left[
\Delta_r(w),\frac{w^2}{4}-Q
\right]$, we observe that $z>w$ and
$z>\frac{r-6}{4}w$.
Moreover, since $Q>0$, we have
$z\le
\frac{w^2}{4}-Q
<
\frac{w^2}{4}$.
According to the above bounds on $z$, we obtain the following estimates:
$$
\frac{3}{w+z+3}<\frac{3}{2w},
\qquad
-\frac{2}{\frac w2+2-\frac{2z}{w}}
<
-\frac4w, \qquad \frac{4(r-5)}
{(r-3)w+\frac{4(r-5)}{r-6}z}
<
\frac2w,
$$
where the last inequality holds as 
\[
(r-3)w+\frac{4(r-5)}{r-6}z
>
2(r-4)w.
\]
Consequently, a straightforward computation gives
\[
\begin{aligned}
H_{w,r}'(z)
=
\frac{3}{w+z+3}
+
\frac{4(r-5)}
{(r-3)w+\frac{4(r-5)}{r-6}z}
-
\frac{2}{\frac w2+2-\frac{2z}{w}}<
\frac{3}{2w}+\frac2w-\frac4w
=
-\frac1{2w}<0.
\end{aligned}
\]
Thus $H_{w,r}(z)$ is strictly decreasing on
$\left[
\Delta_r(w),\frac{w^2}{4}-Q
\right]$. Since
$2+\frac{2Q}{w}
=
\frac w2+2-\frac{2}{w}
\left(\frac{w^2}{4}-Q\right)$,
it follows from \eqref{eq8}, \eqref{eq14} and Claim \ref{xshxsh-1} that
\begin{equation}\label{eq18}
\ln\varphi(a_1,\ldots,a_r)
\le
H_{w,r}\left(\frac{w^2}{4}-Q\right)\le H_{w,r}\bigl(\Delta_r(w)\bigr).
\end{equation}

We now prove that
$H_{w,r}\bigl(\Delta_r(w)\bigr)<g(w)$ with
$w\ge4r$. By the definition of $H_{w,r}$,
\begin{equation}\label{eq19a}
\begin{aligned}
H_{w,r}\bigl(\Delta_r(w)\bigr)
-\left(
(w+3)\ln\left(\frac w2+2\right)-\ln4
\right)
=&
3\ln\left(
\frac{w+\Delta_r(w)+3}
{3(\frac w2+2)}
\right)\\
&+
(r-6)\ln\left(
1+\frac{4(r-5)\Delta_r(w)}
{(r-3)(r-6)w}
\right)\\
&+
w\ln\left(
1-\frac{4\Delta_r(w)}{w(w+4)}
\right)
+\ln4.
\end{aligned}
\end{equation}
Since $\Delta_r(w)=
\frac{r-4}{2}w-\frac{r^2}{4}+r+1$, we estimate the three logarithmic terms in \eqref{eq19a} separately. For the first term, since
$$
\frac{r-2}{3}
-
\frac{w+\Delta_r(w)+3}
{3(\frac w2+2)}
=
\frac{(r-4)(r+8)}{6(w+4)}>0,
$$
we have
\begin{equation}\label{eq19b}
3\ln\left(
\frac{w+\Delta_r(w)+3}
{3(\frac w2+2)}
\right)
<
3\ln\frac{r-2}{3}.
\end{equation}
For the second term, we have
$$
\begin{aligned}
3+\frac{4}{(r-3)(r-6)}
-
\left(
1+\frac{4(r-5)\Delta_r(w)}
{(r-3)(r-6)w}
\right)=
\frac{(r-5)(r^2-4r-4)}
{w(r-3)(r-6)}>0.
\end{aligned}
$$
Thus, combining with Lemma \ref{xsh} (4), we have
\begin{equation}\label{eq19c}
\begin{aligned}
(r-6)\ln\left(
1+\frac{4(r-5)\Delta_r(w)}
{(r-3)(r-6)w}
\right)
&<
(r-6)\ln\left(
3+\frac{4}{(r-3)(r-6)}
\right)\\
&=
(r-6)\ln3
+
(r-6)\ln\left(
1+\frac{4}{3(r-3)(r-6)}
\right)\\
&\leq(r-6)\ln3+\frac{4}{3(r-3)}.
\end{aligned}
\end{equation}
For the last term, since
$$
\frac{\mathrm{d}}{\mathrm{d}w}
\left(
\frac{4\Delta_r(w)}{w+4}
\right)
=
\frac{r^2+4r-36}{(w+4)^2}>0
$$
and $w\ge4r$, we obtain
$$
\frac{4\Delta_r(w)}{w+4}
\ge
\frac{4\Delta_r(4r)}{4r+4}
=
\frac{7r^2-28r+4}{4(r+1)}.
$$
Moreover, by \eqref{eq14} and $Q>0$, we have
$
0<
\frac{4\Delta_r(w)}{w(w+4)}
<
1.
$
Hence, by Lemma \ref{xsh} (2),
\begin{equation}\label{eq19e}
\begin{aligned}
w\ln\left(
1-\frac{4\Delta_r(w)}{w(w+4)}
\right)
\le
-\frac{4\Delta_r(w)}{w+4}\le
-\frac{7r^2-28r+4}{4(r+1)}.
\end{aligned}
\end{equation}

Combining \eqref{eq19a}, \eqref{eq19b}, \eqref{eq19c}, and
\eqref{eq19e}, we obtain
\begin{equation}\label{eq19}
\begin{aligned}
&H_{w,r}\bigl(\Delta_r(w)\bigr)
-\left(
(w+3)\ln\left(\frac w2+2\right)-\ln4
\right)\\
&\quad<
3\ln\frac{r-2}{3}
+(r-6)\ln3
+\frac{4}{3(r-3)}
+\ln4
-\frac{7r^2-28r+4}{4(r+1)}.
\end{aligned}
\end{equation}
Denote the right-hand side of \eqref{eq19} by $F(r)$.
It remains only
to prove that $F(r)<-1/4$ for $r\ge7$. For $r=7$ and $r=8$, a direct computation gives
\[
F(7)\approx-0.3680<-\frac14,
\qquad
F(8)\approx-0.4037<-\frac14.
\]
For  $r\ge8$, differentiating $F(r)$ gives
\begin{equation}\label{eq19f}
\begin{aligned}
F'(r)
&=
\frac3{r-2}
+\ln3
-\frac{4}{3(r-3)^2}
-\frac74
+\frac{39}{4(r+1)^2}<
\frac12+\frac{11}{10}
-\frac74+\frac{13}{108}=
-\frac4{135}<0.
\end{aligned}
\end{equation}
Hence $F(r)$ is strictly decreasing for $r\ge8$. Together with the
values of $F(7)$ and $F(8)$, we obtain
$F(r)<-\frac14 $ for $r\ge7$. It follows from \eqref{eq19} that
\[
H_{w,r}\bigl(\Delta_r(w)\bigr)
<
(w+3)\ln\left(\frac w2+2\right)-\ln4-\frac14
=
g(w).
\]
Together with \eqref{eq18}, this yields
$\ln\varphi(a_1,\ldots,a_r)<g(w)$. By Lemma \ref{lem:T3lower},
\[
\ln\varphi(a_1,a_2,\ldots,a_r)
<g(w)<\ln T_3(w)=\ln\varphi(\alpha,\beta,\gamma),
\]
which contradicts the fact that
$(x_1,x_2,\ldots,x_r)
=(a_1,a_2,\ldots,a_r)$
maximizes $\varphi$ over $\mathcal{A}_w^{d-1}$.
This completes the proof.
\end{proof}

\begin{lemma}\label{xh71}
$r\notin\{7,8,\ldots,d-1\}$ with $r>w/4$.
\end{lemma}
\begin{proof}
Suppose, to the contrary, that
$r\in\{7,8,\ldots,d-1\}$ with $r>w/4$, and write
$(x_1,x_2,\ldots,x_r)=(a_1,a_2,\ldots,a_r)$.
Since $a_i\ge1$ and $\sum_{i=1}^r a_i=w$, we have $r\le w$. By the AM--GM inequality, we obtain
\begin{equation}\label{AM_GM}
\prod_{i=1}^r(1+a_i)
\le
\left(1+\frac wr\right)^r.
\end{equation}
Let $y=\frac rw$. Then
$\frac14<y\le1$. Moreover, by \eqref{eq9},
\begin{equation}\label{mmm}
\begin{aligned}
2+\frac{2Q}{w}
\le
2+\frac2w
\left(
\frac{w^2}{4}
-\frac{r-4}{2}w
+\frac{r^2}{4}
-r-1
\right)=
6-2y-\frac2w+\frac w2(1-y)^2.
\end{aligned}
\end{equation}
Therefore, by \eqref{eq8}, \eqref{AM_GM} and \eqref{mmm}, we have
\begin{equation}\label{eq21}
\ln\varphi(a_1,\ldots,a_r)\le K_w(y),
\end{equation}
where
\begin{equation*}
K_w(y)
:=
wy\ln\left(1+\frac1y\right)
+w\ln\left(
6-2y-\frac2w+\frac w2(1-y)^2
\right).
\end{equation*}

We first consider $11\le w\le27$. Since $r\ge7$, only finitely many
integer pairs $(w,r)$ satisfy
$11\le w\le27$ and $7\le r\le w$. By executing the following
\texttt{Mathematica} \cite{Mathematica} command:
\begin{lstlisting}[language=Mathematica, frame = trBL, aboveskip=10pt, belowskip=10pt, numbers=none, rulecolor=\color{black}, literate={TableForm}{{\textcolor{blue}{TableForm}}}9]
TableForm@
 Flatten[Table[
   N[r*Log[1 + w/r] + 
     w*Log[6 - 2*r/w - 2/w + w/2 (1 - r/w)^2] - (w + 3) Log[w/2 + 2] +
      Log[4] + 1/4, 100], {w, 11, 27}, {r, 7, w}]]
\end{lstlisting}
one can verify that, for all such pairs,
\[
K_w\left(\frac rw\right)
-
\left(
(w+3)\ln\left(\frac w2+2\right)-\ln4
\right)
<
-\frac14.
\]
Hence, by \eqref{eq21},
\begin{equation}\label{eq23}
\ln\varphi(a_1,\ldots,a_r)<g(w)
\qquad
(11\le w\le27).
\end{equation}

It remains to consider $w\ge28$. We first show that $K_w(y)$ is
strictly decreasing on $[1/4,1]$. Since $1/4\le y\le1$, we have
$6-2y-\frac2w+\frac w2(1-y)^2
\ge
4-\frac2w>0$.
Moreover,
\begin{equation}\label{eq22a}
\frac1wK_w'(y)
=
\ln\left(1+\frac1y\right)
-\frac1{1+y}
-
\frac{w(1-y)+2}
{6-2y-\frac2w+\frac w2(1-y)^2}.
\end{equation}
By Lemma \ref{xsh} (4),
\begin{equation}\label{eq22b}
\ln\left(1+\frac1y\right)-\frac1{1+y}
\le
\frac1{2y(1+y)}.
\end{equation}
It follows from \eqref{eq22a} and \eqref{eq22b} that it suffices
to prove
\begin{equation}\label{eq22c}
2y(1+y)\bigl(w(1-y)+2\bigr)
>
6-2y-\frac2w+\frac w2(1-y)^2.
\end{equation}
Let
$D_w(y) \coloneqq
2y(1+y)\bigl(w(1-y)+2\bigr)
-\left(
6-2y-\frac2w+\frac w2(1-y)^2
\right)$.
Then
\[
D_w''(y)=8-w-12wy<0
\qquad
\left(\frac14\le y\le1,\ w\ge28\right),
\]
so $D_w$ is strictly concave on $[1/4,1]$. Moreover, for $w\ge 28$, we observe that
\[
D_w\left(\frac14\right)
=
\frac{3w^2-68w+32}{16w}
>0,
\qquad
D_w(1)=4+\frac2w>0.
\]
Hence $D_w(y)>0$ for every $y\in[1/4,1]$, and therefore
\eqref{eq22c} holds. Consequently,
$K_w'(y)<0$ for
$\frac14\le y\le1$ and $w\ge28$.

Since $y>1/4$, it follows from \eqref{eq21}  that
$\ln\varphi(a_1,\ldots,a_r)
\le
K_w(y)
<
K_w\left(\frac14\right)$.
Hence
\begin{equation}\label{eq24a}
\begin{aligned}
&\ln\varphi(a_1,\ldots,a_r)
-\left(
(w+3)\ln\left(\frac w2+2\right)-\ln4
\right)\\
&\quad<
\Theta(w) \coloneqq
\frac w4\ln5
+w\ln\left(
\frac{9w^2+176w-64}{16w(w+4)}
\right)
-3\ln\left(\frac w2+2\right)
+\ln4.
\end{aligned}
\end{equation}
Differentiating $\Theta(w)$ and applying $\ln x\le x-1$ to the
logarithmic term gives
\begin{equation}\label{eq24b}
\Theta'(w)
\le
\frac14\ln5-\frac7{16}
-
\frac{353w^3-4416w^2+3264w-1024}
{4w(w+4)(9w^2+176w-64)}.
\end{equation}
Observe that
$\frac14\ln5<\frac7{16}$
and for $w\ge28$,
\[
\begin{aligned}
&353w^3-4416w^2+3264w-1024\\
&\quad=
353(w-28)^3
+25236(w-28)^2
+586224(w-28)
+4377280>0 .
\end{aligned}
\]
Then \eqref{eq24b} yields
$\Theta'(w)<0$, that is, $\Theta(w)$ is decreasing for $w\ge28$.
Therefore
\[
\Theta(w)\le\Theta(28)=
7\ln5
+28\ln\frac{745}{896}
-10\ln2 <
-\frac14.
\]
Together with \eqref{eq24a}, this gives
\begin{equation}\label{eq24}
\ln\varphi(a_1,\ldots,a_r)<g(w)
\qquad
(w\ge28).
\end{equation}

Combining \eqref{eq23} and \eqref{eq24}, we conclude that
$\ln\varphi(a_1,\ldots,a_r)<g(w)$
for every $w\ge11$. By Lemma \ref{lem:T3lower},
\[
\ln\varphi(a_1,a_2,\ldots,a_r)
<g(w)<\ln T_3(w)=\ln\varphi(\alpha,\beta,\gamma),
\]
which contradicts the fact that
$(x_1,x_2,\ldots,x_r)
=(a_1,a_2,\ldots,a_r)$
maximizes $\varphi$ over $\mathcal{A}_w^{d-1}$.
Therefore
$r\notin\{7,8,\ldots,d-1\}$ when $r>w/4$.
\end{proof}

By Lemmas \ref{xh4}--\ref{xh71}, we conclude that $r=3$, that is, the set $S$ contains exactly three consecutive integers. Assume that $S=\{i-1,i,i+1\}$ for $3\leq i\leq d-1$.
We are now in a position to complete the proof of Theorem \ref{main_th2}.

\begin{proof}[\em\textbf{Proof of Theorem \ref{main_th2}}]
Recall that $S=\{i-1,i,i+1\}$. Set
\[
a=n_{i-1}-1,\qquad
b=n_i-1,\qquad
c=n_{i+1}-1,
\qquad
x=a+c.
\]
Then $a,b,c\geq1$ and $a+b+c=w$.
Moreover, by the definition of $\varphi$,
\begin{equation}\label{eq25}
\begin{aligned}
\tau(G^*)
=\varphi(a,b,c)
&=(a+1)(b+1)(c+1)
(b+2)^{a+c}(a+c+2)^b\\
&=(a+1)(b+1)(c+1)
(b+2)^x(x+2)^b.
\end{aligned}
\end{equation}

First we prove that $|a-c|\leq1$. Suppose, without loss of generality,
that $a\geq c+2$. Then
$(a-1,b,c+1)\in\mathcal{A}_w^3
\subseteq\mathcal{A}_w^{d-1}$.
By the definition of $\varphi$,
\[
\frac{\varphi(a-1,b,c+1)}{\varphi(a,b,c)}
=
\frac{a(c+2)}{(a+1)(c+1)}
=
1+\frac{a-c-1}{(a+1)(c+1)}
>1.
\]
This contradicts the fact that $(a,b,c)$ maximizes $\varphi$ over
$\mathcal{A}_w^{d-1}$. Hence $|a-c|\leq1$.

Let
$p(t) \coloneqq \left\lfloor\frac{(t+2)^2}{4}\right\rfloor$.
Since $a+c=x$ and $|a-c|\leq1$, we have $(a+1)(c+1)=p(x)$.
Thus \eqref{eq25} becomes
\begin{equation}\label{eq26}
\varphi(a,b,c)
=
p(x) \cdot (b+1)(b+2)^x(x+2)^b.
\end{equation}
We claim that $x-b\geq0$. Suppose, to the contrary, that
$b\geq x+1$. Choose $a', c' \in \mathbb{N}^+$ such that
$a'+c'=x+1$ and $|a'-c'|\leq1$. Then
$(a',b-1,c')\in\mathcal{A}_w^3
\subseteq\mathcal{A}_w^{d-1}$.
By \eqref{eq26} and the definition of $p$,
\[
\begin{aligned}
\frac{\varphi(a',b-1,c')}{\varphi(a,b,c)}
=
\frac{p(x+1) \cdot b(b+1)^{x+1}(x+3)^{b-1}}
{p(x) \cdot (b+1)(b+2)^x(x+2)^b}=
\frac{p(x+1)}{p(x)}
\frac{b}{x+2}
\left(\frac{b+1}{b+2}\right)^x
\left(\frac{x+3}{x+2}\right)^{b-1}.
\end{aligned}
\]
Since $b\geq x+1$, we have $b-1\geq x$ and
$\frac{b+1}{b+2}\geq\frac{x+2}{x+3}$.
Therefore
\[
\left(\frac{b+1}{b+2}\right)^x
\left(\frac{x+3}{x+2}\right)^{b-1}
\geq1.
\]
Since $\frac{b}{x+2}\geq\frac{x+1}{x+2}$, we have
\[
\frac{\varphi(a',b-1,c')}{\varphi(a,b,c)}
\geq
\frac{p(x+1)}{p(x)}
\frac{x+1}{x+2}.
\]
We claim that
\[
\frac{p(x+1)}{p(x)}
\frac{x+1}{x+2}>1.
\]
Indeed, if $x=2s$, then
\[
\frac{p(x+1)}{p(x)}
\frac{x+1}{x+2}
=
\frac{(s+1)(s+2)}{(s+1)^2}
\cdot
\frac{2s+1}{2s+2}
=
\frac{(s+2)(2s+1)}{2(s+1)^2}
>1,
\]
while if $x=2s+1$, then
\[
\frac{p(x+1)}{p(x)}
\frac{x+1}{x+2}
=
\frac{(s+2)^2}{(s+1)(s+2)}
\cdot
\frac{2s+2}{2s+3}
=
\frac{2(s+2)}{2s+3}
>1.
\]
Thus
$\varphi(a',b-1,c')>\varphi(a,b,c)$, contradicting the maximality of $(a,b,c)$. Hence $x-b\geq0$.

We finally show that $x-b\leq1$. Suppose, to the contrary, that
$x-b\geq2$. Choose $a'', c'' \in \mathbb{N}^+$ such that
$a''+c''=x-1$ and $|a''-c''|\leq1$.
Then
$(a'',b+1,c'')\in\mathcal{A}_w^3
\subseteq\mathcal{A}_w^{d-1}$.
By \eqref{eq26} and the definition of $p$,
\[
\begin{aligned}
\frac{\varphi(a'',b+1,c'')}{\varphi(a,b,c)}
&=
\frac{p(x-1) \cdot (b+2)(b+3)^{x-1}(x+1)^{b+1}}
{p(x) \cdot (b+1)(b+2)^x(x+2)^b}\\
&=
\frac{p(x-1)}{p(x)}
\frac{x+1}{b+1}
\left(\frac{b+3}{b+2}\right)^{x-1}
\left(\frac{x+1}{x+2}\right)^b.
\end{aligned}
\]
Since $x\geq b+2$, we have $x-1\geq b+1$ and
$\frac{x+1}{x+2}\geq\frac{b+3}{b+4}$.
Therefore
\[
\begin{aligned}
\left(\frac{b+3}{b+2}\right)^{x-1}
\left(\frac{x+1}{x+2}\right)^b\geq
\left(\frac{b+3}{b+2}\right)^{b+1}
\left(\frac{b+3}{b+4}\right)^b=
\frac{b+3}{b+2}
\left(
\frac{(b+3)^2}{(b+2)(b+4)}
\right)^b
>1.
\end{aligned}
\]
Also, since $b+1\leq x-1$, we have
$\frac{x+1}{b+1}\geq\frac{x+1}{x-1}$.
Thus
\[
\frac{\varphi(a'',b+1,c'')}{\varphi(a,b,c)}
>
\frac{p(x-1)}{p(x)}
\frac{x+1}{x-1}.
\]
We claim that
\[
\frac{p(x-1)}{p(x)}
\frac{x+1}{x-1}>1.
\]
Indeed, if $x=2s$, then
\[
\frac{p(x-1)}{p(x)}
\frac{x+1}{x-1}
=
\frac{s(s+1)}{(s+1)^2}
\cdot
\frac{2s+1}{2s-1}
=
\frac{s(2s+1)}{(s+1)(2s-1)}
>1,
\]
while if $x=2s+1$, then
\[
\frac{p(x-1)}{p(x)}
\frac{x+1}{x-1}
=
\frac{(s+1)^2}{(s+1)(s+2)}
\cdot
\frac{2s+2}{2s}
=
\frac{(s+1)^2}{s(s+2)}
>1.
\]
Hence $\varphi(a'',b+1,c'')>\varphi(a,b,c)$,
 contradicting the maximality of $(a,b,c)$. Thus
$x-b\leq1$.

Combining the two inequalities gives $0\leq x-b\leq1$. Since $x$ and $b$ are integers, $x-b\in\{0,1\}$.
Together with $x+b=w$, this yields
$b=\left\lfloor\frac{w}{2}\right\rfloor$ and
$x=\left\lceil\frac{w}{2}\right\rceil$.
Since $a+c=x$ and $|a-c|\leq1$, we further obtain
$\{a,c\}
=
\left\{
\left\lfloor\frac{x}{2}\right\rfloor,
\left\lceil\frac{x}{2}\right\rceil
\right\}
=
\left\{
\left\lfloor\frac{w+1}{4}\right\rfloor,
\left\lceil\frac{w}{4}\right\rceil
\right\}$.
Recalling that $a=n_{i-1}-1$,
$b=n_i-1$ and
$c=n_{i+1}-1$,
we conclude that
$n_i=\left\lfloor\frac{w}{2}\right\rfloor+1$
and $n_j=1$
for all $j\notin\{i-1,i,i+1\}$, while
$\{n_{i-1},n_{i+1}\}
=
\left\{
\left\lfloor\frac{w+1}{4}\right\rfloor+1,
\left\lceil\frac{w}{4}\right\rceil+1
\right\}$.

Substituting these values into \eqref{eq1}, we obtain
\[
\begin{aligned}
\tau(G^*)
=
\left(\left\lfloor\frac{w+1}{4}\right\rfloor+1\right)
\left(\left\lceil\frac{w}{4}\right\rceil+1\right)
\left(\left\lfloor\frac{w}{2}\right\rfloor+1\right)
\left(\left\lfloor\frac{w}{2}\right\rfloor+2\right)^{
\left\lceil\frac{w}{2}\right\rceil}
\left(\left\lceil\frac{w}{2}\right\rceil+2\right)^{
\left\lfloor\frac{w}{2}\right\rfloor}.
\end{aligned}
\]
This proves the stated upper bound and characterizes all graphs attaining
equality. 
\end{proof}

\section{Concluding remarks}
The proofs of Lemmas \ref{xh1}--\ref{xh3} remain valid for every $w\geq1$. Hence, for $1\leq w\leq10$, the same structural reduction shows that the problem reduces to maximizing $\varphi$ over $\mathcal{A}_w^{d-1}$. Since $\mathcal{A}_w^{d-1}$ is finite for each $w \le 10$, the problem of maximizing the number of spanning trees in $\mathcal{B}(n, d)$ for $d \ge 4$ and $n - d - 1 = w \le 10$ can be resolved computationally using the \texttt{Python} script given in Appendix \ref{sc_code}. The complete solution to the considered extremal problem is thus summarized in Table \ref{summary_tab}. In particular, for $d \ge 4$ and $w \in \{ 3, 4, 5, \ldots, 9 \}$, the solution has the same form as in Theorem \ref{main_th2} for $d \ge 4$ and $w \ge 11$. On the other hand, if $d\geq5$ and $w=10$, then the graphs maximizing the number of spanning trees in $\mathcal{B}(n, d)$ are precisely those of the form
\[
P_{d+1}\circ\left(
\underbrace{1,1,\ldots,1}_{\mbox{\scriptsize $t$ times}},
2,5,5,2,
\underbrace{1,1,\ldots,1}_{\mbox{\scriptsize $d-3-t$ times}}
\right)
\]
for some $t\in\{1,2,\ldots,d-4\}$. This extremal structure is different from that for $w\geq11$.

\begin{table}[h]
\centering
\begin{tabular}{|c|c|}
\hline
Case & Extremal graph(s)\\
\hline
\hline
$d = 2$ & $K_{\lfloor \frac{n}{2} \rfloor, \lceil \frac{n}{2} \rceil}$\\
\hline
$d = 3$ & $K_{\lfloor \frac{n}{2} \rfloor, \lceil \frac{n}{2} \rceil}'$\\
\hline
$d \ge 4$ and $w = 0$ & $P_{d + 1}$\\
\hline
$d \ge 4$ and $w = 1$ & $P_{d + 1} \circ (1, \ldots, 1, 2, 1, \ldots, 1)$\\
\hline
$d \ge 4$ and $w = 2$ & $P_{d + 1} \circ (1, \ldots, 1, 2, 2, 1, \ldots, 1)$\\
\hline
$d \ge 4$, $w \ge 3$ and $w \neq 10$ & \multirow{2}{*}{$P_{d + 1} \circ (1, \ldots, 1, \lfloor \frac{w + 1}{4} \rfloor + 1, \lfloor \frac{w}{2} \rfloor + 1, \lceil \frac{w}{4} \rceil + 1, 1, \ldots, 1)$}\\
\cline{1--1}
$d = 4$ and $w = 10$ & \\
\hline
$d \ge 5$ and $w = 10$ & $P_{d + 1} \circ (1, \ldots, 1, 2, 5, 5, 2, 1, \ldots, 1)$\\
\hline
\end{tabular}
\caption{Extremal graphs maximizing the number of spanning trees in $\mathcal{B}(n,d)$.}
\label{summary_tab}
\end{table}

In this paper, we solved the spanning tree maximization problem for connected bipartite graphs with prescribed order and diameter. For $d\ge4$ and $w\ge11$, every extremal graph has exactly
three consecutive nontrivial entries, with the extra vertices
distributed approximately in the ratio $1:2:1$.
In this range, the maximum spanning tree number depends on
$n$ and $d$ only through $w=n-d-1$.
When $w=10$ and $d\ge5$, an exceptional structure with four
consecutive nontrivial parts occurs.

Since spanning trees represent minimal connected spanning backbones of a network, the obtained extremal graphs can be viewed as bipartite network topologies with the largest connectivity redundancy measured by the number of spanning trees under the prescribed order and diameter constraints. From the viewpoint of network reliability, this provides a structural characterization of bipartite networks that maximize the number of minimal connected operating states. Moreover, because a larger number of spanning trees gives a network more alternative ways to remain connected after edge failures, our result also indicates enhanced connectivity robustness of the extremal topologies.

The present work focuses on bipartite network models. A natural and interesting direction is to consider general network models under the same order and diameter constraints. Let $\mathcal{G}(n,d)$ be the set of all connected graphs of order $n$ and diameter $d$. This leads to the following problem.

\begin{problem}
Determine the maximum number of spanning trees over all graphs in $\mathcal{G}(n,d)$, together with the corresponding extremal graphs.
\end{problem}

\section*{Declarations}

\noindent\textbf{Generative AI and AI-assisted technologies.}
AI assistance was used in the function comparisons in the proofs of Lemmas 4.7 and 4.8. All AI-generated output was
critically reviewed and verified by the authors.
\smallskip

\noindent\textbf{Funding.}
The work was supported by the National Natural Science Foundation of China (Grant No.\ 12271251), Postgraduate Research \& Practice Innovation Program of Jiangsu Province, grant number KYCX25\_0625, the Ministry of Science, Technological Development and Innovation of the Republic of Serbia, grant number 451-03-34/2026-03/200102, and the Science Fund of the Republic of Serbia, grant \#6767, Lazy walk counts and spectral radius of threshold graphs --- LZWK.
\smallskip

\noindent\textbf{Conflict of interest.}
The authors declare that they have no conflict of interest.
\smallskip

\noindent\textbf{Data and code availability.}
No datasets were generated or analyzed during the current study. The \texttt{Python} script used to computationally verify some of the results in Table \ref{summary_tab} is given in Appendix \ref{sc_code}.

\appendix
\section{\texttt{Python} script}\label{sc_code}

\begin{lstlisting}[language=Python, frame = trBL, aboveskip=10pt, belowskip=10pt, numbers=left, rulecolor=\color{black}]
from itertools import combinations


def compute_phi(vector):
    r"""
    This function computes \varphi(a_1, a_2, \ldots, a_r) for a given tuple
    (a_1, a_2, \ldots, a_r).
    """

    result = 1
    for i in range(len(vector)):
        neighbor_sum = 2
        if i > 0:
            neighbor_sum += vector[i - 1]
        if i < len(vector) - 1:
            neighbor_sum += vector[i + 1]

        for _ in range(vector[i]):
            result *= neighbor_sum

        result *= vector[i] + 1

    return result


def resolve_w(w, max_r=None):
    r"""
    This function finds the extremal tuples (a_1, a_2, \ldots, a_r) with respect
    to \varphi for a given w. The argument ``max_r``, if not `None`, limits the
    number of entries allowed in the tuples. Its default value is `None`.
    """

    best_phi = -1
    best_vectors = []
    max_r = max_r if max_r is not None else w + 1

    for slices in range(1, max_r + 1):
        # The increasing tuples of r - 1 integers from 1, 2, \ldots, w - 1 give
        # rise to tuples of r integers that sum up to w.
        for inc in combinations(range(1, w), slices - 1):
            # Extract the underlying tuple (a_1, a_2, \ldots, a_r).
            if slices > 1:
                vector = [0] * slices
                vector[0] = inc[0]
                vector[-1] = w - inc[-1]
                for i in range(1, slices - 1):
                    vector[i] = inc[i] - inc[i - 1]
            else:
                vector = [w]

            # Compute \varphi(a_1, a_2, \ldots, a_r).
            phi_value = compute_phi(vector)

            # Update the list of best tuples.
            if phi_value > best_phi:
                best_phi = phi_value
                best_vectors = [vector]
            elif phi_value == best_phi:
                best_vectors.append(vector)

    return best_vectors, best_phi


def check_uniqueness(best_vectors):
    """
    This function confirms that a provided list of tuples contains only one tuple
    up to reflection, and then returns that tuple.
    """

    assert len(best_vectors) == 1 or len(best_vectors) == 2
    if len(best_vectors) == 2:
        first = best_vectors[0]
        second = best_vectors[1][::-1]
        assert first == second

    return tuple(best_vectors[-1])


def check_favorite_form(vector):
    r"""
    This function confirms that the provided tuple has the extremal form described
    in the paper, i.e., the form (1) for w = 1, (1, 1) for w = 2, and (a, b, c)
    with b = \lfloor w/2 \rfloor and 0 \le a - c \le 1, for w \ge 3.
    """

    w = sum(vector)

    if w == 1:
        assert vector == (1,)
    elif w == 2:
        assert vector == (1, 1)
    else:
        assert w >= 3 and len(vector) == 3 and vector[1] == w // 2
        rest = (w + 1) // 2
        assert vector[0] == (rest + 1) // 2 and vector[2] == rest // 2


if __name__ == "__main__":
    print("Solution for d = 4 and w = 10:", end=" ", flush=True)

    best_vectors, _ = resolve_w(w=10, max_r=3)
    solution = check_uniqueness(best_vectors)
    print(f"{solution}.")
    check_favorite_form(solution)

    print("Now, assume that d >= 4 for any w, except w = 10, where we assume that d >= 5.")

    for w in range(1, 21):
        print(f"w = {w}:", end=" ", flush=True)

        best_vectors, _ = resolve_w(w)
        solution = check_uniqueness(best_vectors)
        print(f"{solution}.")

        # The same extremal form is attained for every w, except for w = 10.
        if w != 10:
            check_favorite_form(solution)

print("Done!")
\end{lstlisting}

\end{document}